\documentclass[12pt,letterpaper]{article}

\newcommand{\required}[1]{\section*{\hfil \sharp1\hfil}}

\usepackage{verbatim,url}

\usepackage{amssymb,amsmath,amsfonts,amsthm}
\usepackage{hyperref}
\usepackage{upref}
\newcommand{\Beq}{\begin{equation}}
\newcommand{\Eeq}{\end{equation}}
\newcommand{\beq}{\begin{equation*}}
\newcommand{\eeq}{\end{equation*}}
\newcommand{\bal}{\begin{align}}
\newcommand{\eal}{\end{align}}
\newcommand{\G}{\Gamma}

\newcommand{\D}{\mathrm{d}}
\renewcommand{\O}{\Omega}

\newcommand{\lb}{\left(}

\newcommand{\rb}{\right)}

\renewcommand{\O}{\Omega}

\newcommand{\F}{f_{i_1\dots i_m}}

\newtheorem{theorem}{Theorem}
\newtheorem{Lemma}{Lemma}
\newtheorem{prop}{Proposition}
\newtheorem{definition}{Definition}

\theoremstyle{definition}
\newtheorem{claim}{Claim}

\newtheorem{remark}{Remark}

\newcommand{\ac}[1]{\begin{quotation}\textbf{Anuj's comment:\
		}{\textit{\sharp1}}\end{quotation}}
\newcommand{\rc}[1]{\begin{quotation}\textbf{Rohit's comment:\
		}{\textit{\sharp1}}\end{quotation}}

\title{ Support Theorems and an Injectivity Result for Integral Moments of a Symmetric $m$-Tensor Field}
\date{}
\author{Anuj Abhishek and Rohit Kumar Mishra}

\begin{document}
	\maketitle
	\begin{abstract}
	In this work, we show an injectivity result and support theorems for integral moments of an $m$-tensor field on a simple, real analytic, Riemannian manifold. Integral moments of $m$-tensor fields were first introduced by Sharafutdinov. First we generalize a Helgason type support theorem proven by Krishnan and Stefanov in ``A support theorem for the geodesic ray transform of symmetric tensor fields", Inverse Problems and Imaging, 3(3):453-464,2009. We use this extended result along with the first $(m+1)$-integral moments of an $m$-tensor field to prove the aforementioned results.
	\end{abstract}
	\section{Introduction}
	 Let $(\O,g)$ be a compact, simple, real-analytic Riemannian manifold of dimension $n$ with smooth boundary. We will parametrize the maximal geodesics in $\O$ with endpoints on $\partial \O$ by their starting points and directions.
	 \par  Set 
	 $$\G_{-}:=\big{\{}(x,\xi)\in T\O|\ x\in \partial \O, |\xi|=1,\langle\xi,\nu(x)\rangle < 0\big{\}},$$
	 where $\nu(x)$ is the outer unit normal to $\partial \O$ at $x$.  Then we will define the $q$-th integral moment of a symmetric $m$-tensor field $f$, $I^qf$ as a function on $\G_{-}$ by
	  \begin{align*}
	 I^qf(x, \xi) = \int_{0}^{l(\gamma_{x,\xi})} t^q\langle f(\gamma_{x,\xi}(t)), \dot{\gamma}_{x,\xi}^m(t) \rangle dt = \int_{0}^{l(\gamma_{x,\xi})}t^q f_{i_1\dots i_m}(\gamma_{x,\xi}(t))\dot{\gamma}_{x,\xi}^{i_1}(t)\cdots \dot{\gamma}_{x,\xi}^{i_m}(t) dt.
	 \end{align*}
	 where $\gamma_{x,\xi}(t)$ is the geodesic starting from $x$ in the direction $\xi$ and $l(\gamma_{x,\xi})$ is the value of  the parameter $t$ at which this geodesic intersects the boundary again. The above definition of integral moments for a symmetric $m$-tensor fields was first introduced by Sharafutdinov in the context of $\mathbb{R}^n$, see \cite{Sharafutdinov_Generalized_Tensor_Fields}. In the same paper he proved that if the first $(m+1)$ integral moments $I^qf$ for $q=0,1, \dots, m$ of a compactly supported symmetric $m$-tensor field $f$ are known along all straight lines, then $f$ can be uniquely recovered.
	 \par \noindent Zeroth integral moment coincides with the usual geodesic ray transform of a symmetric $m$-tensor field. In this work, we are interested in  injectivity results and  support theorem for integral moments defined above. Microlocal techniques play a very crucial role in proving such results. Guillemin  first introduced the microlocal approach in the Radon transform setting, see \cite{Guillemin-Sternberg_Book}. Analytic microlocal techniques were used by Boman and Quinto in \cite{Boman-Quinto-Duke} to prove support theorems for Radon transforms with positive real-analytic weights. For more literature on such support theorems, we refer to the reader \cite{Q2006:supp,Boman1991,Boman1993,Boman1992,Eskin1998,Gonzalez1994,Sato1973,Zhou2000} and references therein. For the analytic microlocal techniques used in this paper, we will mainly refer to \cite{SU2,SU1,SU3,Sjostrand,KS}.
	 \par \noindent The geodesic ray transform of any symmetric tensor field of order 2, which in our notation will be denoted by $I^{0}(f)$ arises naturally in the context of lens and boundary rigidity problems and has been studied in e.g. \cite{Sharafutdinov_Book},\cite{Sharafutdinov},\cite{SU1},\cite{SU2}. Support theorems for such transforms have been of independent interest among mathematicians. In  \cite{SU2}, the authors prove a s-injectivity result for symmetric $2$ tensors fields. The same proof works for a symmetric tensor field of any order. That is if  $I^{0}(f) =0$ for  all the geodesics of $\O$ then its solenoidal part vanishes. A question arises as to what data is sufficient for us to conclude such an injectivity result for the tensor field $f$ itself. Using the result stated above, we show that if  $I^q f =0$ for $q = 0, 1, \dots , m$ for all the geodesics of $\O$, then $f=0$. Injectivity result for the local geodesic ray transform of a function has been proved in \cite{Uhlmann2016} using new techniques. We also treat the case in which the integral moments are known for the open set of geodesics that do not intersect a given geodesically convex set. We do so using the techniques laid out in \cite{KS}, where the authors prove a Helgason type support theorem for symmetric tensor fields of order $2$ over simple, real analytic Riemannian manifolds. We first extend the result in \cite{KS} for symmetric $m$ tensor fields. Using this new result, we prove a stronger version of such support type theorems, i.e. if we know $I^q f =0$ for $q = 0, 1, \dots , m$ over the open set of geodesics not intersecting a convex set, then it implies that the support of $f$ lies in the convex set. We would also like to mention that Krishnan already proved such a support theorem for the case of functions in \cite{K1}.
	 \par\noindent The paper is organized as follows. In Section $2$ we give the definitions and our main theorems. Section $3$ has some preliminary propositions and lemmas that are needed for the proof of the main theorems. In Section $4$ we will prove a Helgason type support theorem which we state in Section $2$ and prove the support theorem. In Section $5$ we prove the s-injectivity result mentioned above and use it to prove the injectivity of integral moments. We will provide proof of some lemmas and inequalities in the Appendix.
	\par \noindent \textbf{Acknowledgements :} We would like to thank Vladimir Sharafutdinov for suggesting this problem. Besides, we would also like to express our sincere gratitude to Eric Todd Quinto and Venky Krishnan for several hours of fruitful discussions. The authors benefited from the support of the Airbus Group Corporate Foundation
	Chair ``Mathematics of Complex Systems" established at TIFR Centre for Applicable Mathematics
	and TIFR International Centre for Theoretical Sciences, Bangalore, India. 
	\section{Definitions and Statements of the Theorems}
	\begin{definition}[Simple Manifold]
		A compact Riemannian manifold $(\O,g)$ with boundary is said to be simple if 
		\begin{enumerate}
			\item[(i)] The boundary $\partial \O$ is strictly convex: $\langle \nabla_\xi \nu(x) , \xi \rangle > 0 $ for each $\xi \in T_x(\partial \O)$ where $\nu(x)$ is the unit outward normal to the boundary. 
			\item[(ii)] The map $\exp_x:\exp_x^{-1}( \O) \rightarrow \O$ is a diffeomorphism for each $x \in \O$.
		\end{enumerate}
	\end{definition}
	\noindent The second condition ensures that any two points $x, y$ in $\O$ are connected by a unique geodesic in $\O$ that depends smoothly on $x, y$. Any simple manifold $M$ is necessarily diffeomorphic to a ball in $\mathbb{R}^n$, see \cite{Sharafutdinov_Book}.  Therefore, in the analysis of simple manifolds, we can assume that $\O$ is a domain $\Omega \subset \mathbb{R}^n$.
	We are going to work on a fixed simple Riemannian manifold $(\O,g)$ with a fixed real analytic atlas. A  tensor field is said to be analytic on a set $U$ if it is real analytic in some neighborhood of $U$. Let $S^m(\O)$ be the collection of symmetric $m$-tensor fields defied on $\O$. We will work with symmetric $m$-tensor field $f = \{f_{i_1\dots i_m}\}$. We will assume the  Einstein summation convention and raise and lower indexes using the metric tensor. The tensors $\F$ and $f^{i_1\dots i_m} =  f_{j_1\dots j_m}g^{i_1j_1}\cdots g^{i_mj_m}$ will be thought of as same tensors with different representations. 
	\par \noindent It is well known from \cite{Sharafutdinov_Book} that any symmetric $m$-tensor field can be decomposed uniquely in the following way:
	\begin{theorem}\label{Th:Decomposition}
		\cite[Theorem 3.3.2]{Sharafutdinov_Book} \ Let $\O$ be a compact Riemannian manifold with boundary; let $k\geq  1$
		and $m\geq 0$ be integers. For every field $f \in H^k(S^m(\O))$, there exist uniquely determined $f^s \in H^k(S^m(\O))$ and $v\in H^{k+1}(S^{m-1}(\O))$ such that
		$$ f =f^s + \D v,\quad \quad  \delta f^s = 0,\quad \quad  v|_{\partial \O} = 0.$$
		We call the fields $f^s$ and $\D v$ the solenoidal and potential parts respectively of the field $f$.
	\end{theorem}

	\par \noindent Let $\widetilde{\O}$ be an open, real analytic extension of $\O$ such that  $g$ can also be extended to a real analytic metric in  $\widetilde{\O}$. We will also extend all symmetric tensor fields $f$ defined on $\O$ by $0$ in $\widetilde{\O}\setminus  \O$.
	We will think of  each maximal geodesic in $\O$ as a restriction of a geodesic with distinct endpoints in $\widetilde{\O}\setminus \O$ to $\O$. Let $\gamma_{[x,y]}$ be the geodesic connecting   $x$ and $y$.
	\par \noindent Let $\mathcal{A}$ be an open set of geodesics with endpoints in $\widetilde{\O} \setminus \O$ such that any geodesic in $\mathcal{A}$ is homotopic, within the set  $\mathcal{A}$, to a geodesic lying outside $\O$. Set of points lying on the geodesics in $\mathcal{A}$ is denoted by $\O_{\mathcal{A}}$ i.e. 	$\O_{\mathcal{A}} = \underset{\gamma \in \mathcal{A}}{\bigcup} \gamma$ and $\partial_{\mathcal{A}} \O =\O_{\mathcal{A}} \cap \partial \O $. Now we will define what we mean by a geodesically convex subset.
	\begin{definition}
		A subset $K$ of the Riemannian manifold $(\O,g)$ is said to be geodesically convex if for any two points $x\in K$ and $y\in K$, the geodesic connecting them lies entirely in the set $K$.
	\end{definition}
	\par \noindent Finally, let  $\mathcal{E}^\prime(\widetilde{\O})$ be the space of compactly supported tensor fields. We can then extend the definition of $I$ by duality on tensor fields which are distributions in $\widetilde{\O}$ supported in $\O$, see \cite{KS}. Now we are ready to state the main theorems that we will prove in this article.
	
	\begin{theorem}\label{Th:Extension theorem}
		Let $f$ be a symmetric $m$-tensor field on a simple real analytic manifold $(\O,g)$ with components in $\mathcal{E}^\prime(\widetilde{\O})$ and supported in $\Omega$ and $K$ be a closed geodesically convex subset of $\O$. If for each geodesic $\gamma$ not intersecting $K$, we have that $I^{0}f(\gamma) = 0$ then we can find a $(m-1)$-tensor field $v$ with components in $ \mathcal{D}^\prime (\text{int}(\widetilde{\O})\setminus K)$  such that $f = \D v$ in $\text{int}(\widetilde{\O})\setminus K$ and $v = 0$ in $\text{int}(\widetilde{\O})\setminus \O$.	
	\end{theorem}
\noindent Here we would like to mention that this theorem has been shown to be true for the case $m=2$ in \cite{KS}.
	\begin{theorem}\label{Th:main Theorem}
		Let $f$ be a symmetric $m$-tensor field on a simple real analytic manifold $(\O,g)$ with components in $\mathcal{E}^\prime(\widetilde{\O})$ and supported in $\Omega$ and $K$ be a closed geodesically convex subset of $\O$. If for each geodesic $\gamma$ not intersecting $K$, we have that $I^{q}f(\gamma) = 0$ for $q = 0,1,\dots ,m$ then $supp(f) \subset K$.
	\end{theorem}

	\begin{theorem}\label{Th:Injectivity}
		Let $(\O,g)$ be a simple real analytic manifold and $g$ is real analytic in a neighborhood of $\text{cl}({\O})$. If for a symmetric $m$-tensor field $f$ with components in $L^2(\O)$, we have that $I^qf = 0$ for $q = 0,1,\dots ,m$. Then $f=0$.
	\end{theorem}
	\noindent Here we would like to comment that the Theorem \ref{Th:Injectivity} also follows as a corollary of Theorem \ref{Th:main Theorem} when $f$ is supported in $\O$, however as we show in this paper that it can also be proved independently using $s$-injectivity of ray transform where we say $I^0 =I$ is $s$-injective if $If=0$ implies $f^s=0$. In the next section we will prove a proposition and some lemmas that will be needed for the proofs of our main theorems.
	 \section{Preliminaries}\label{Section: Preliminaries}
	 
 \par \noindent We will now prove some results which are analogue of some results already proved for the case of symmetric 2-tensor fields in \cite{KS}. These will be needed later in the proof of our main theorems. 
 \newline Fix a maximal geodesic $\gamma_0$ connecting $x_0 \neq y_0$ in the closure of $\widetilde{\O}$. We construct normal coordinates $x = (x^\prime, x^n)$ at $x_0$ in $\widetilde{\O}$ so that $x^n$ is the distance to $x_0$, and $\frac{\partial}{\partial x^n}$ is normal to $\frac{\partial}{\partial x^\alpha}$, $\alpha <n$, see \cite[Section 2]{SU1}. In these coordinates, the metric $g$ satisfies $g_{ni} = \delta_{ni}$, for all $i$, and the Christoffel symbols satisfy $\Gamma^i_{nn} = \Gamma_{in}^n= 0$. Under these coordinates lines of the type $x^\prime = \text{constant}$ are now geodesics with $x^n$ as arc length parameter. 
\par \noindent Let $U$ be a tubular neighborhood of $\gamma_0$ in $\O$, $U =\{(x^\prime,x^n):|x^\prime|<\epsilon$, $a(x^\prime)\leq x^n \leq b(x^\prime)\}$,
where $\partial \O$ is locally given by $x^n = a(x^\prime)$ and $x^n = b(x^\prime)$. In the next proposition, we prove that  for a symmetric $m$- tensor field $f$, one can always construct an $(m-1)$-tensor field $v$ in $U$ such that for
$$h :=f - \D v$$
one has 
$$ h_{i_1\dots i_{m-1}n} =  0, \quad \text{ for all possible values of }i_j \text{ and } v(x^\prime,a(x^\prime))= 0.$$ 
 $\widetilde{U}$ denotes the tubular neighborhood of $\gamma_0$ of the same type but in $\widetilde{\O}$.
 \begin{remark}
 Numbers of $n$ in the suffix of the tensor $v_{n\dots ni_1\dots i_k}$ will be clear from the order of the tensor $v$. For example, if $ v$ is a $m$- tensor then 
 $$ v_{n\dots ni_1\dots i_k} = v_{\underbrace{n\dots n}_{m-k-times}i_1\dots i_k}.$$
 \end{remark}
	 \begin{prop}\label{prop:Construction of ODE}
	 	Let $f$ be a symmetric $m$-tensor field then there exists a unique $(m-1)$-tensor field $v$ such that $v(x^\prime,a(x^\prime))= 0$ and for $h =f -\D v$, we have 
	 	$$ h_{i_1\dots i_{m-1}n} =  0, \quad \text{ for all possible values of }i_j.$$ 
	 \end{prop}
	 To prove this proposition, we need the following lemma for which we provide a proof in the Appendix section:
	 \begin{Lemma}\label{Lemma:expansion of dv}
	 	Let $ v$ be a symmetric $(m-1)$-tensor field. Then for any $0\leq k\leq m$, we have 
	 	\begin{align*}
	 	(\D v)_{n\dots n i_k\dots i_1} &=\frac{(m-k)}{m}\frac{\partial v_{n\dots ni_k\dots i_1}}{\partial x^n} -\frac{2(m-k)}{m} \sum_{l=1}^k\Gamma_{i_l n}^p v_{n\dots n i_k\dots \hat{i_l}\dots i_1 p}\\ 
	 	&\quad +\frac{1}{m}\sum_{l=1}^k \frac{\partial v_{n\dots ni_k\dots \hat{i_l}\dots  i_1}}{\partial x^{i_l}}-\frac{2}{m} \sum_{l,q=1, l \neq q}^k\Gamma_{i_l i_q}^p v_{n\dots n i_k\dots \hat{i_l}\dots\hat{i_q}\dots i_1 p}.
	 	\end{align*}
	 \end{Lemma}
	 
	 Now, let us come back to the proof of Proposition \ref{prop:Construction of ODE}.
	 \begin{proof}[Proof of Proposition \ref{prop:Construction of ODE}]
	 	Let us first recall the following definition:
	 	\begin{align*}
	 	(\D v)_{i_1\dots i_m} = \sigma(i_1,\dots,i_m)\left(\frac{\partial v_{i_1\dots i_{m-1}}}{\partial x^{i_m}}-\sum_{l=1}^{m-1}\Gamma_{i_m i_l}^pv_{i_1,\dots i_{l-1} p i_{l+1}\dots i_{m-1}}\right)
	 	\end{align*}
	 	where $\sigma$ is a the symmetrization operator.\\
	 	Proving $$ h_{i_1\dots i_{m-1}n} =  0$$ is equivalent to proving the existence of a $(m-1)-$tensor field $v$ such that
	 	\begin{align*}
 (\D v)_{i_1\dots i_{m-1}n}&= f_{i_1\dots i_{m-1}n}.
	 	\end{align*}
	 	First we consider
	 	\begin{align*}
	\frac{\partial v_{n\dots n}}{\partial x^n} = f_{n\dots n}.
	 	\end{align*}
	 	 We will solve this equation together with the initial condition $	v_{n\dots n}(x^\prime,a(x^\prime))= 0$ to get $v_{n\dots n}$.
	 After solving for $v_{n\dots n}$ we will consider
	 \begin{align*}
 \hspace{5cm} (\D v)_{n\dots n i} &=  f_{n\dots n i} \\
	\Rightarrow \hspace{1.9cm}	\frac{\partial v_{ n\dots n i}}{\partial x^n}(x) - 2\Gamma_{in}^p v_{n\dots n p}(x) &=\frac{m}{m-1} f_{ n\dots n i}(x)-\frac{1}{m-1}\frac{\partial v_{n\dots n}}{\partial x^{i}}(x).
	 \end{align*}
	 Now we will solve this system of equations together with the initial conditions $v_{n\dots n i}(x^\prime,a(x^\prime)) = 0$ to get $v_{n\dots n i}$.
	 	 	\par Proceeding in a similar manner let us assume that for a given $k$ such that $0 \leq (k-1) \leq (m-1)$, we have already found $v_{n\dots ni_{k-1}\dots i_1}$ for which $h_{n\dots ni_{k-1}\dots i_1} = f_{n\dots ni_{k-1}\dots i_1} - (\D v)_{n\dots ni_{k-1}\dots i_1} = 0$. If $(k-1) = (m-1)$ then we are done and if not then we can find $v_{n\dots ni_k\dots i_1}$ in the following manner. Using Lemma \ref{Lemma:expansion of dv}, we can construct the following system of equations for $h_{n\dots ni_{k}\dots i_1} = 0$.
	 \begin{align*}
	 	\frac{\partial v_{n\dots ni_k\dots i_1}}{\partial x^n}(x) -2 \sum_{l=1}^k\Gamma_{i_l n}^p v_{n\dots n i_k\dots \hat{i_l}\dots i_1 p}(x) &=	\frac{1}{(m-k)}\left\{ m f_{n\dots ni_k\dots i_1}(x)-\sum_{l=1}^k \frac{\partial v_{n\dots ni_k\dots \hat{i_l}\dots  i_1}}{\partial x^{i_l}}(x)\right.\\
	 	&\quad +\left. 2 \sum_{l,q=1, l \neq q}^k\Gamma_{i_l i_q}^p v_{n\dots n i_k\dots \hat{i_l}\dots\hat{i_q}\dots i_1 p}(x)\right\}.
	 \end{align*}
	 Finally, we will solve the above system of equations with the initial conditions 	$v_{n\dots ni_k\dots i_1}(x^\prime,a(x^\prime)) = 0$ to get $v_{n\dots ni_k\dots i_1}$ uniquely. We repeat the same process till $k=(m-1)$ to prove the proposition.
 \end{proof}
\begin{Lemma}\label{Lemma:Different definition of v}
Let $f$ be supported in $\O$, and $I^{0}f(\gamma) = 0$ for all maximal geodesics in $\widetilde{U}$ belonging to some neighborhood of the geodesics $x_0 = const$. Then $v = 0$ in $int(\widetilde{U})\setminus \O$.
\end{Lemma}
\begin{proof}
First let $f\in C^{\infty}(\O)$. We will give another invariant definition of $v$ and use it to conclude our lemma. For any $x \in \widetilde{U}$ and any $\xi \in T_x\widetilde{U}\setminus \{0\}$ so that $\gamma_{x,\xi}$ stays in $\widetilde{U}$, we set
\begin{align}\label{eq:1}
	u(x,\xi) = \int_0^{l(x,\xi)} f_{i_1\dots i_m}(\gamma_{x,\xi}(t))\dot{\gamma}_{x,\xi}^{i_1}(t)\cdots \dot{\gamma}_{x,\xi}^{i_m}(t)dt.
\end{align}
Extend the definition of $\gamma_{x, \xi}$ for $\xi \neq 0$ as a solution of the geodesic equation. Then $u(x, \xi)$ is homogeneous of order $(m-1)$ in $\xi$. Consider 
\begin{align*}
	u(x,\lambda\xi) &= \lambda^{m-1}u(x,\xi)\\
	\Rightarrow\hspace{10mm} \xi^{j_1}\cdots \xi^{j_{m-1}}\frac{\partial^{m-1}}{\partial\xi^{j_1}\cdots \partial\xi^{j_{m-1}}}u(x,\lambda\xi)&= (m-1)!\ u(x,\xi), \quad \text{diff. }(m-1)\text{  times w.r.t }\lambda\\
	\Rightarrow\hspace{13mm} \xi^{j_1}\cdots \xi^{j_{m-1}}\frac{\partial^{m-1}}{\partial\xi^{j_1}\cdots \partial\xi^{j_{m-1}}}u(x,\xi)&= (m-1)!\ u(x,\xi), \text{ for } \lambda =1.
\end{align*}
Now, we shall define a symmetric $(m-1)$- tensor field $v$ as following:
\begin{align}\label{eq:2}
	v_{i_1\dots i_{m-1}}(x) = \frac{1}{(m-1)!}\left.\frac{\partial^{m-1}}{\partial\xi^{i_1}\cdots \partial\xi^{i_{m-1}}}u(x,\xi)\right|_{\xi = e_n}.
\end{align}
Consider for any $0\leq l\leq (m-1)$
\begin{align*}
	v_{i_1\dots i_{m-1-l}n\dots n}(x) &= \frac{1}{(m-1)!}\left.\frac{\partial^{m-1}}{\partial\xi^{i_1}\cdots \partial\xi^{i_{m-1-l}}\partial \xi^n\cdots \partial \xi^n}u(x,\xi)\right|_{\xi = e_n}\\
	&= \frac{1}{(m-1)!}\left.\xi^{j_1}\cdots \xi^{j_{l}}\frac{\partial^{m-1}}{\partial\xi^{i_1}\cdots \partial\xi^{i_{m-1-l}}\partial \xi^{j_1}\cdots \partial \xi^{j_l}}u(x,\xi)\right|_{\xi = e_n}\\
	&= \frac{l!}{(m-1)!}\left.\frac{\partial^{m-1}}{\partial\xi^{i_1}\cdots \partial\xi^{i_{m-1-l}}}u(x,\xi)\right|_{\xi = e_n}\quad (\text{ using homogenity of } u).
\end{align*}
Then, we have $$ v_{n\dots n}(x) = u(x, e_n).$$
\par \noindent We will now show that with this definition of $v$, for $h= f -\D v$, one has 
$$ h_{i_1\dots i_{m-1}n} =  0, \quad \text{ for all possible values of }i_j .$$
Define \begin{align}\label{eq:3}
	w(x,\xi) = \int_0^{l(x,\xi)} h_{i_1\dots i_m}(\gamma_{x,\xi}(t))\dot{\gamma}_{x,\xi}^{i_1}(t)\cdots \dot{\gamma}_{x,\xi}^{i_m}(t)dt.
\end{align} 

\begin{claim}\label{claim:1}
	Let $0 \leq l \leq (m-1)$ and $w(x,\xi)$ is defined as above then
	\begin{equation}\label{eq:4}
		\left.\frac{\partial^{l}}{\partial\xi^{j_1}\cdots \partial\xi^{j_{l}}}w(x,\xi)\right|_{\xi =e_n} = 0.
	\end{equation}
\end{claim}
\begin{proof} Consider for any $ 0\leq l\leq (m-1)$,
	\begin{align*}
		&\frac{\partial^{l}}{\partial\xi^{j_1}\cdots \partial\xi^{j_{l}}}w(x,\xi)\\
		&= \frac{\partial^{l}}{\partial\xi^{j_1}\cdots \partial\xi^{j_{l}}}u(x,\xi) 
		- \frac{\partial^{l}}{\partial\xi^{j_1}\cdots \partial\xi^{j_{l}}}\int_0^{l(x,\xi)} (\D v)_{i_1\dots i_m}(\gamma_{x,\xi}(t))\dot{\gamma}_{x,\xi}^{i_1}(t)\cdots \dot{\gamma}_{x,\xi}^{i_m}(t)dt\\
		&= \frac{\partial^{l}}{\partial\xi^{j_1}\cdots \partial\xi^{j_{l}}}u(x,\xi) 
		- \frac{\partial^{l}}{\partial\xi^{j_1}\cdots \partial\xi^{j_{l}}}\int_0^{l(x,\xi)}\frac{d}{dt} \left(v_{i_1\dots i_{m-1}}(\gamma_{x,\xi}(t))\dot{\gamma}_{x,\xi}^{i_1}(t)\cdots \dot{\gamma}_{x,\xi}^{i_{m-1}}(t)\right)dt\\
		&= \frac{\partial^{l}}{\partial\xi^{j_1}\cdots \partial\xi^{j_{l}}}u(x,\xi) 
		- \frac{\partial^{l}}{\partial\xi^{j_1}\cdots \partial\xi^{j_{l}}} \left(v_{i_1\dots i_{m-1}}(x)\xi^{i_1}\cdots \xi^{i_{m-1}}\right)\\
		&= \frac{\partial^{l}}{\partial\xi^{j_1}\cdots \partial\xi^{j_{l}}}u(x,\xi) 
		- \frac{(m-1)!}{l!}\  v_{j_1\dots j_{l}n\dots n}(x)
	\end{align*}
	\begin{align*}
		\Rightarrow\hspace{25mm} \left.\frac{\partial^{m-1}}{\partial\xi^{j_1}\cdots \partial\xi^{j_{l}}}w(x,\xi)\right|_{\xi =e_n} &= \left.\frac{\partial^{l}}{\partial\xi^{j_1}\cdots \partial\xi^{j_{l}}}u(x,\xi)\right|_{\xi =e_n} - \frac{(m-1)!}{l!}\  v_{j_1\dots j_{l}n\dots n}(x)\\
		&= \frac{(m-1)!}{l!}\ v_{j_1\dots j_{l}n\dots n}(x) - \frac{(m-1)!}{l!}\  v_{j_1\dots j_{l}n\dots n}(x)\\
		&=0.
	\end{align*}
\end{proof}
\noindent Now let us recall the following relation \cite[Section 1.2]{Sharafutdinov_Book}
\begin{align}\label{eq:5}
	Gw(x, \xi) =h_{i_1\dots i_m}(x)\xi^{i_1}\cdots \xi^{i_m}
\end{align}
where $G= \xi^i \partial_{x^i} - \Gamma^k_{ij}\xi^i\xi^j\partial_{\xi^k} $ is the generator of the geodesic flow. After differentiating \eqref{eq:5} $(m-1)$ times w.r.t. $\xi$, we get
\begin{align*}
	\frac{\partial^{m-1}}{\partial\xi^{j_1}\cdots \partial\xi^{j_{m-1}}}Gw (x,\xi)  = m!\ h_{j_1\dots j_{m-1}i}(x)\xi^i\\
\Rightarrow\hspace{20mm} \left.\frac{\partial^{m-1}}{\partial\xi^{j_1}\cdots \partial\xi^{j_{m-1}}}Gw (x,\xi)\right|_{\xi =e_n}  = m!\ h_{j_1\dots j_{m-1}n}(x).
\end{align*}
We will prove L.H.S. of the above equation is $0$. This will prove our lemma. Consider
\begin{align*}
	\frac{\partial Gw(x,\xi)}{\partial \xi^{j_1}}  &= \frac{\partial}{\partial \xi^{j_1}} \left(\xi^i \frac{\partial}{\partial x^i} w(x,\xi)\right) - \Gamma^k_{ij}\frac{\partial}{\partial \xi^{j_1}} \left(\xi^i\xi^j \frac{\partial}{\partial \xi^{k}}w(x,\xi)\right)\\
	&= \frac{\partial w(x,\xi)}{\partial x^{j_1}} + \xi^i \frac{\partial^2 w(x,\xi)}{\partial \xi^{j_1}\partial x^i}- \Gamma^k_{ij}\frac{\partial}{\partial \xi^{j_1}}(\xi^i\xi^j) \frac{\partial w(x,\xi)}{\partial \xi^{k}} - \Gamma^k_{ij}\xi^i\xi^j\frac{\partial^2 w(x,\xi)}{ \partial \xi^{j_1}\partial \xi^{k}}\\
	\Rightarrow\hspace{10mm} \frac{\partial^2 Gw(x,\xi)}{\partial \xi^{j_1}\partial \xi^{j_2}}  &= \frac{\partial^2 w(x,\xi)}{\partial x^{j_1}\partial \xi^{j_2}} + \frac{\partial^2 w(x,\xi)}{\partial \xi^{j_1}\partial x^{j_2}}+\xi^i \frac{\partial^3 w(x,\xi)}{\partial \xi^{j_1}\partial \xi^{j_2}\partial x^i}- \Gamma^k_{ij}\frac{\partial^2}{\partial \xi^{j_1}\partial \xi^{j_2}}(\xi^i\xi^j) \frac{\partial w(x,\xi)}{\partial \xi^{k}} \\
	&\quad - \Gamma^k_{ij}\frac{\partial}{\partial \xi^{j_1}}(\xi^i\xi^j)\frac{\partial^2 w(x,\xi)}{ \partial \xi^{j_2}\partial \xi^{k}}-\Gamma^k_{ij}\frac{\partial}{\partial \xi^{j_2}}(\xi^i\xi^j) \frac{\partial^2 w(x,\xi)}{ \partial \xi^{j_1} \partial \xi^{k}} - \Gamma^k_{ij}\xi^i\xi^j\frac{\partial^3 w(x,\xi)}{ \partial \xi^{j_1}\partial \xi^{j_2}\partial \xi^{k}}\\
	&= \frac{\partial^2 w(x,\xi)}{\partial x^{j_1}\partial \xi^{j_2}} + \frac{\partial^2 w(x,\xi)}{\partial \xi^{j_1}\partial x^{j_2}}+\xi^i \frac{\partial^3 w(x,\xi)}{\partial \xi^{j_1}\partial \xi^{j_2}\partial x^i}- 2\Gamma^k_{j_1j_2} \frac{\partial w(x,\xi)}{\partial \xi^{k}} \\
	&\quad - 2\Gamma^k_{ij_1}\xi^i\frac{\partial^2 w(x,\xi)}{ \partial \xi^{j_2}\partial \xi^{k}}-  2\Gamma^k_{ij_2}\xi^i\frac{\partial^2 w(x,\xi)}{ \partial \xi^{j_1}\partial \xi^{k}}- \Gamma^k_{ij}\xi^i\xi^j\frac{\partial^3 w(x,\xi)}{ \partial \xi^{j_1}\partial \xi^{j_2}\partial \xi^{k}}.
\end{align*}
Using similar calculations, we get
\begin{align*}
	\frac{\partial^{m-1} Gw (x,\xi)}{\partial\xi^{j_1}\cdots \partial\xi^{j_{m-1}}}&= \xi^i\frac{\partial^{m} w (x,\xi)}{\partial\xi^{j_1}\cdots \partial\xi^{j_{m-1}}\partial x^i}- \sum_{l,k=1,l\neq k}^{m-1}2\Gamma_{j_l j_p}^k\frac{\partial^{m-2}w (x,\xi) }{\partial \xi^{k}\partial\xi^{j_1}\cdots \partial \hat{\xi^{j_l}}\cdots \partial \hat{\xi^{j_p}}\cdots \partial\xi^{j_{m-1}}}\\
	&\quad + \sum_{l=1}^{m-1}\frac{\partial^{m-1}w (x,\xi) }{\partial x^{j_l}\partial\xi^{j_1}\cdots \partial \hat{\xi^{j_l}}\cdots \partial\xi^{j_{m-1}}}
	- \sum_{l=1}^{m-1}2\Gamma_{i j_l}^k \xi^i\frac{\partial^{m-1}w (x,\xi) }{\partial \xi^{k}\partial\xi^{j_1}\cdots \partial \hat{\xi^{j_l}}\cdots \partial\xi^{j_{m-1}}} \\
	&\quad  -\Gamma_{ij}^k \xi^i \xi^j \frac{\partial^{m}w (x,\xi) }{\partial \xi^{k}\partial\xi^{j_1}\cdots \partial\xi^{j_{m-1}}}.
\end{align*}
Which implies
$$\left.\frac{\partial^{m-1} Gw (x,\xi)}{\partial\xi^{j_1}\cdots \partial\xi^{j_{m-1}}}\right|_{\xi =e_n}=0, \quad \quad (\text{Using Claim \ref{claim:1} and } \Gamma_{nn}^k =0).$$
Now that we have proved the proposition for the case when $f$ is smooth, it can be extended to the case when $f$ is a distribution by exactly the same reasoning as in \cite [Lemma 3.1]{KS}.
\end{proof}

	\section{Proofs of Theorem \ref{Th:Extension theorem} and Theorem \ref{Th:main Theorem}}
We will start with proving some lemmas and propositions required to prove our main theorems.
\begin{Lemma}\label{Lemma:Wavefront}
	Let $f$ be a symmetric $m$-tensor field as above. Let $\gamma_0$ be a geodesic of $\widetilde{\Omega}$ and $U$ be a neighborhood of $\gamma_{0}$ in $\widetilde{\O}$. Assume that $\text{WF}_{A}(f)\bigcap \pi^{-1}(\text{U })$ does not contain co-vectors of the type $(\xi^{\prime},0)$, then $h=f-\D v$ also does not contain such co-vectors.
\end{Lemma}

\begin {proof}
Since $v$ and $\D v$ have the same analytic wavefront set, so we will prove the lemma for $v$. We will prove this by induction by proving it for $v_{n\dots n i_{k}\dots i_1}$ for every $k\leq (m-1)$.
Let us first do the analysis for $v_{n\dots n}$. Note that $v_{n\dots n}$ can be rewritten as a convolution with the Heaviside function in the following manner 
\begin{align*}
	v_{n\dots n}(x) &= \int_{-\infty}^{x^n}f_{n\dots n}(x^{'},y^n)dy^n\\
	&= \int_{-\infty}^{\infty}f_{n\dots n}(x^{'},y^n)H(x^n-y^n)dy^n.
\end{align*}
The wavefront set of the convolution can be found by applying \cite[8.2.16]{H1}. Since we have assumed that  $\text{WF}_{A}(f)\bigcap \pi^{-1}(\text{U })$ does not contain co-vectors of the type $(\xi^{\prime},0)$, hence it will be true for $v_{n\dots n}(x)$ as well. 
Now let us assume that the lemma holds for any $ 0 \leq k-1 < (m-1)$ i.e. $v_{n\dots n i_{k-1}\dots i_1}$ satisfies the same wavefront conditions. We will show that this implies that the Lemma \ref{Lemma:Wavefront} is true for $k$. For this consider the system of ODEs from Lemma \ref{Lemma:expansion of dv},

\begin{align*}
	\frac{\partial v_{n\dots ni_k\dots i_1}}{\partial x^n}(x) -2 \sum_{l=1}^k\Gamma_{i_l n}^p v_{n\dots n i_k\dots \hat{i_l}\dots i_1 p}(x) &=	\frac{1}{(m-k)}\left\{ m f_{n\dots ni_k\dots i_1}(x)-\sum_{l=1}^k \frac{\partial v_{n\dots ni_k\dots \hat{i_l}\dots  i_1}}{\partial x^{i_l}}(x)\right.\\
	&\quad +\left. 2 \sum_{l,q=1, l \neq q}^k\Gamma_{i_l i_q}^p v_{n\dots n i_k\dots \hat{i_l}\dots\hat{i_q}\dots i_1 p}(x)\right\},\\
	v_{n\dots ni_k\dots i_1}(x^\prime,a(x^\prime)) &= 0.
\end{align*}
This can be rewritten as : 
\begin{align*}
	\partial_{n}(\tilde{v}) -A(x^{\prime},x^{n})\tilde{v} &=	w,\\
	\tilde{v}|_{x^n<< 0} &= 0
\end{align*}
where $A$ is an analytic matrix, $\tilde{v}=v_{n\dots n i_k \dots i_1}$ and $\text{WF}_{A}(w)\bigcap \pi^{-1}(\text{U})$ does not have covectors of the type $(\xi^{\prime},0)$. By Duhamel's principle the solution to the above is given by:
\begin{center}
	$\tilde{v}(x^\prime, x^n)=\int_{-\infty}^{x^n}\Phi(x^{\prime},x^{n},y^{n})w(x^{\prime},y^{n})dy^{n}$
\end{center}
where $\Phi$ is analytic. The expression given above for $\tilde{v}(x^\prime, x^n)$ can be rewritten as:  
\begin{center}
	$\tilde{v}(x^\prime, x^n)=\int_{R^n}\Phi(x^{\prime},x^{n},y^{n})H(x^{n}-y^{n})\delta(x^{\prime}-y^\prime)w(y^{\prime},y^{n})dy^{\prime}dy^{n}.$
\end{center} 
The kernel of the integral operator is given by : $\Phi(x^{\prime},x^{n},y^{n})H(x^n-y^n)\delta(x^{\prime}-y^{\prime})$. Note that the frequency set of the analytic wavefront set of the Heaviside and delta distributions here are perpendicular to each other and hence satisfy H\"ormander's non cancellation condition \cite[8.5.3]{H1}. The lemma then follows from the argument in \cite{KS}.
\end{proof}

\subsection{Analyticity along Conormal Directions} 
Before moving further, we will need the following proposition which is an analogue of Proposition 2 from \cite{SU3} and generalizes that proposition for the case when $f$ is a symmetric $m$-tensor. We will mimic the proof for the case when $m=2$ as given in that paper and adapt the arguments wherever necessary to make it work for a symmetric tensor field of any order.
\begin{prop}\label{S-U analogue}
Let $\O$ and $f$ be as above. Let $\gamma_{0}$ be a fixed  geodesic through $x_{0}$ normal to $\xi_{0}$ where $(x_{0},\xi_{0})\in T^{*}\O\backslash 0$. Assume $(I^{0}f)(\gamma)=0$ for all $\gamma$ in a neighborhood of $\gamma_{0}$ and $g$ is analytic in this neighborhood. Let $\delta f = 0$ near $x_{0}$.
Then $$(x_{0},\xi_{0}) \notin \text{WF}_{A}(f).$$ 
\end{prop}
\begin{proof} 
 For the given geodesic $\gamma_0$ that passes through $x_0$ and is normal to $\xi_0$, let us consider a tubular neighborhood $U$ of $\gamma_{0}$ endowed with analytic semi-geodesic coordinates $x=(x^{\prime},x^{n})$ on it. Without loss of generality, assume that $x_0=0$. Furthermore, $\forall x \in \gamma_0$, $x^\prime = 0$. Note that $U$ =\{$(x^{\prime},x^{n}):|x^\prime|<\epsilon \text{ and } l^{-}<x_{n}<l^{+}\text{; }0<\epsilon<<1$\} in this co-ordinate system. Choose $\epsilon$ such that $\{x : x_{n}=l^{-},l^{+} \text{ and } |x^\prime|<\epsilon\}$ lies outside $\O$. Clearly  $\xi_{0}=(\xi_{0}^\prime,0)$. Hence our goal is now to show: $$(0,\xi_{0})\notin \text{WF}_{A}(f).$$
As stated earlier, we will reproduce the arguments from \cite{SU3} here for the sake of completeness.
Consider $Z=\{|x|<\frac{7\epsilon}{8}: |x_{n}|=0\}$ and let $x^\prime$ variable be denoted on $Z$ by $z^\prime$. Then $(z^\prime,\theta^\prime)$ are local co-ordinates in nbd$(\gamma_0)$ (in the set of geodesics) given by $(z^\prime,\theta^\prime)\to \gamma_{(z^\prime,0),(\theta^\prime,1)}$. Here, $|\theta^\prime|<<1$ (where, the geodesic is in the direction $(\theta^\prime,1)$). By following their arguments verbatim, we get the following equation: 


\begin{align}\label{eq:3}
	\int e^{i\lambda z^\prime(x,\theta^\prime).\xi^\prime}a_{N}(x,\theta^\prime)f_{i_{1}\dots i_{m}}(x){b^{i_{1}}}(x,\theta^\prime)\cdots {b^{i_m}}(x,\theta^\prime)dx =0.
\end{align}
Here, $(x,\theta^\prime)\to a_{N}$ is analytic and satisfies 
\begin{align}	\label{eq:estimate on cutoff}
|\partial^{\alpha}a_{N}|\leq(CN)^{|\alpha|}, \quad \alpha \leq N,
\end{align} 
\noindent see \cite[Equation(38)]{SU3}. Also, note that $b(0,\theta^\prime)={\theta}$ and $a_{N}(0,\theta^\prime)=1$.

\noindent Further, let us choose $\theta(\xi)$ to be a vector depending analytically on $\xi$ near $\xi=\xi_{0}$ and satisfying the following conditions:
\begin{align*}
	\theta(\xi)\cdot \xi&=0, \quad {\theta}^{n}(\xi)=1\quad \quad \text{ and }\\
	\theta(\xi_{0})&=(0,\cdots,1)=e_{n}
\end{align*}
Now, we will rewrite \eqref{eq:3} using the above mapping in the following form:
\begin{align}\label{eq:4}
	\int e^{i\lambda \phi(x,\xi)}\tilde{a_{N}}(x,\xi)f_{i_{1}\dots i_{m}}(x){\tilde{b}}^{i_{1}}(x,\xi)\cdots {\tilde{b}}^{i_{m}}(x,\xi)dx =0.
\end{align}
\noindent Here $\phi(x,\xi)=z^\prime\cdot\xi^\prime$. This phase function has been shown in \cite{SU3} to be non-degenerate in a neighborhood of $(0,\xi_{0})$ by showing $\phi_{x\xi}(0,\xi)=$ Id. This also implies that $x\to\phi_{\xi}(x,\xi)$ is a diffeomorphism in this neighborhood.

\noindent To establish the above condition in a neighborhood of the geodesic $\gamma_{0}$, one chooses the co-normal vector
\begin{align}\label{eq: definition of theta}
	\xi_{0}=e_{n-1},\quad \text{i.e. the covector}\  (0,0,\cdots,0,1,0)
\end{align}
and defines $$\theta(\xi)=(\xi_1,\cdots,\xi_{n-2},-\frac{\xi_{1}^{2}+\dots+\xi_{n-2}^{2}+\xi_{n}}{\xi_{n-1}},1).$$ This definition of $\theta $ is consistent with the requirement put on $\theta(\xi)$ as above. 
One can then show that the differential of the map $\xi \to \theta(\xi)$ where $\xi \in S^{n-1}$ is invertible at $\xi_{0}=e_{n-1}$, see \cite[Equation (44)]{SU3}.
\begin{Lemma} \cite[Lemma 3.2]{SU3}\label{Lemma:1}
	Let, ${\theta(\xi)}$ and $\phi(x,\xi)$ be as above. Then, $\exists$ $ \delta>0$ such that if $$\phi_{\xi}(x,\xi)=\phi_{\xi}(y,\xi)$$ for some $x \in U$, $|y|<\delta$, $|\xi-\xi_0|<\delta$ where $\xi$ is complex, then $y=x$.
\end{Lemma}
\noindent We will study the analytic wavefront set of $f$ using Sj\"ostrand's complex stationary phase method. For this assume $x$, $y$ as in Lemma \ref{Lemma:1} and $|\xi_0-\eta|<\frac{\delta}{\tilde{C}}$ with $\tilde{C}>>2$ and $\delta<<1$. Multiply \eqref{eq:4} by
$$\tilde{\chi}(\xi-\eta)e^{i\lambda\left(i\frac{(\xi-\eta)^2}{2}-\phi(y,\xi)\right)}$$
where $\tilde{\chi}$ is the characteristic function of the ball $B(0,\delta)\subset \mathbb{C}^n$ and then integrate w.r.t. $\xi$ to get: 
\begin{align}\label{eq:7}
	\iint e^{i\lambda \Phi(y,x,\xi,\eta)}\tilde{\tilde{a_{N}}}(x,\xi)f_{i_{1}....i_{m}}(z){\tilde{b}}^{i_{1}}(x,\xi)\cdots{\tilde{b}}^{i_{m}}(x,\xi)dxd\xi =0.
\end{align}
In the above equation, $\tilde{\tilde{a_{N}}}=\tilde{\chi}(\xi-\eta)\tilde{a}_{N}$ is another analytic and elliptic amplitude for $x$ close to zero and $|\xi-\eta|<\frac{\delta}{\tilde{C}}$ and $$\Phi=-\phi(y,\xi)+\phi(x,\xi)+\frac{i}{2}(\xi-\eta)^{2}.$$ Furthermore, $$\Phi_{\xi}=\phi_{\xi}(x,\xi)-\phi_{\xi}(y,\xi)+i(\xi-\eta).$$
To apply the stationary phase method we need to know the critical points of $\xi \mapsto \Phi$. Using the Lemma \ref{Lemma:1} above we have:
\begin{enumerate}
	\item[(i)] If $y=x$, $\exists$ a unique real critical point $\xi_{c}=\eta$
	\item[(ii)] If $y \neq x$, there are no real critical points 
	\item[(iii)] Also by Lemma \ref{Lemma:1}, if $y \neq x$, there is a unique complex critical point if $|x-y|<\delta/C_{1}$ and no critical points for $|x-y|>\delta/C_0$ for some constants $C_0 \text{ and }C_1 $ with $C_1 > C_0$.		
\end{enumerate}
Define, $\psi(x,y,\eta):=\Phi(\xi_c)$. Then at $x=y$
\begin{center}
	(i) $\psi_{y}(x,x,\eta) = -\phi_{x}(x,\eta)$ \quad
	(ii) $\psi_{x}(x,x,\eta) = \phi_{x}(x,\eta)$ \quad
	(iii) $\psi(x,x,\eta) =0 $. \quad
\end{center}
Now, we split the $x$ integral in \eqref{eq:7} in to two parts : we integrate over $\{x:|x-y|>\delta/C_{0}\}$ for some $C_{0}>1$ and its complement. Since, $|\Phi_{\xi}|$ has a positive lower bound for $\{x:|x-y|>\delta/C_{0}\}$ and there are no critical points of $\xi \to  \Phi$ in this set, we can estimate that integral in the following manner: First note that, $e^{i\lambda \Phi(x,\xi)} = \frac{\Phi_{\xi}\partial_{\xi}}{i\lambda|\Phi_\xi|^{2}}e^{i\lambda \Phi(x,\xi)}$. Using, \eqref{eq:estimate on cutoff} and integrating by parts $N$ times with respect to $\xi$ and the fact that on the boundary $|\xi-\eta|=\delta$, we get
\begin{align}\label{eq:8}
	\left|	\iint_{|x-y|>\delta/C_{0}} e^{i\lambda \Phi(y,x,\xi,\eta)}\tilde{\tilde{a_{N}}}(x,\xi)f_{i_{1}\dots i_{m}}(x){\tilde{b}}^{i_{1}}(x,\xi)\cdots{\tilde{b}}^{i_{m}}(x,\xi)dxd\xi \right|\leq C\bigg(\frac{CN}{\lambda}\bigg)^{N}+CNe^{-\frac{\lambda}{C}}.
\end{align}

\noindent We choose $N\leq\lambda/Ce\leq N+1$ to get an exponential error on the right. Now in estimating the integral
\begin{align}\label{eq:9}
	\left|	\int_{|x-y|\leq\delta/C_{0}} e^{i\lambda \Phi(y,x,\xi,\eta)}\tilde{\tilde{a_{N}}}(x,\xi)f_{i_{1}\dots i_{m}}(x){\tilde{b}}^{i_{1}}(x,\xi)\cdots{\tilde{b}}^{i_{m}}(x,\xi)dxd\xi \right|,
\end{align}
we use \cite[Theorem 2.8]{Sjostrand} and the \cite[Remark 2.10]{Sjostrand} following that to conclude: 


\begin{align}\label{eq:10}
	\int_{|x-y|\leq\delta/C_{0}}e^{i\lambda \psi(x,\alpha)}f_{i_{1}\dots i_{m}}(x)B^{i_{1}\dots i_{m}}(x,\alpha;\lambda)dx=\mathcal{O}(e^{-\lambda/C})
\end{align}
where  $\alpha=(y,\eta)$ and $B$ is a classical analytical symbol with principal part $\tilde{b}\otimes\cdots \otimes\tilde{b}$. See appendix below for a proof of estimates in (\ref{eq:8}) and (\ref{eq:10}).
\par \noindent Let, $\beta=(y,\mu)$ where, $\mu=\phi_{y}(y,\eta)=\eta+\mathcal{O}(\delta)$. At $y=0$, we have $\mu=\eta$. Also $\alpha \to \beta$ is a diffeomorphism following similar analysis as in \cite[Section 4]{SU3}. If we write $\alpha=\alpha(\beta)$, then the above equation becomes: 
\begin{align}\label{eq:11}
	\int_{|x-y|\leq\delta/C_{0}}e^{i\lambda \psi(x,\beta)}f_{i_{1}\dots i_{m}}(x)B^{i_{1}\dots i_{m}}(x,\beta;\lambda)dx=\mathcal{O}(e^{-\lambda/C})
\end{align}
where $\psi$ satisfies (i), (ii) and (iii), and $B$ is a classical analytical symbol as before and
\begin{center}
	$\psi_{y}(x,x,\eta) = -\mu$,\quad 
	$\psi_{x}(x,x,\eta) = \mu$ \quad and \quad 
	$\psi_{y}(x,x,\eta) =0.$
\end{center}
The symbols in \eqref{eq:11} satisfy : $$\sigma_{P}(B)(0,0,\mu)= {\theta}(\mu)\otimes\cdots \otimes{\theta}(\mu) = \theta^{\otimes m}(\mu)$$
and in particular,  $$\sigma_{P}(B)(0,0,\xi_0)= e_n\otimes \cdots  \otimes e_n.$$ 
\par \noindent Let, ${\theta}_1=e_{n}$, ${\theta}_2,\dots,{\theta}_N$ be $N=\binom{n+m-2}{m}$ unit vectors at $x_0=0$ lie in the hyperplane perpendicular to $\xi_0$. We will also assume that $\{\theta_i^{\odot m}\}_{i=1}^N$ are independent, where $\odot$ is a symmetrized  product of vectors.
Existence of such vectors in any open set in $\xi^{ \perp}_0$ can be shown. We can therefore assume that $\theta_{p}$ belongs in a small neighborhood around $\theta_{1}=e^{n}$.  Then we can rotate the axes a little such that $\xi_0=e^{n-1}$ and $\theta_{p}=e_{n}$ and do the same construction as above. This gives us $N=\binom{n+m-2}{m}$ phase functions $\psi_{(p)}$, and as many number of analytic symbols for which \eqref{eq:11} is true i.e.
\begin{align}\label{eq:12}
	\int_{|x-y|\leq\delta/C_{0}}e^{i\lambda \psi_{(p)}(x,\beta)}f_{i_{1}\dots i_{m}}(x)B_{(p)}^{i_{1}\dots i_{m}}(x,\beta;\lambda)dx=\mathcal{O}(e^{-\lambda/C})
\end{align}
where  $$\sigma_{P}(B_p)(0,0,\mu)= {\theta_p}(\mu)\otimes\cdots\otimes{\theta_p}(\mu),\quad p=1,\dots, N \quad \text{ up to  elliptic factors}.$$
Now we use the fact that $\delta f = 0 $ near $x_{0}$. So integrating 
$$\frac{1}{\lambda}\exp(i\lambda \psi_{(1)}(x,\beta)\chi_{0}\delta f=0$$ w.r.t. $x$ and after an integration by parts, we get
\begin{equation}\label{eq:13}
	\int_{|x-y|\leq\delta/C_{0}}e^{i\lambda \psi_{(1)}(x,\beta)}f_{i_{1}...i_{m}}(x)C^{i_{m}}(x,\beta;\lambda)dx=\mathcal{O}(e^{-\lambda/C})\text{, } i_{j}\in\{1,\dots, n\} \text{ and } j = 1, \dots , (m-1)
\end{equation} 
for $\beta_x=y$ small enough and where $\sigma_P(C^{i_m})(0,0,\xi_0)=(\xi_0)^{i_m}$.
This gives us additional $\tilde{N}=\binom{n+m-2}{m-1}$ equations such that the system of $N+\tilde{N}=\binom{n+m-1}{m}$ equations \eqref{eq:12}, \eqref{eq:13} can be viewed as a tensor valued operator on $f$. We claim that the symbol for this operator  is elliptic at $(0,0,\xi_0)$.
Indeed, to show that the symbol is elliptic at $(0,0,\xi_0)$ amounts to showing that the only solution to following system of equations is $f=0$:
\begin{align}
	\theta_{p}^{i_1}\dots \theta_{p}^{i_m} f_{i_1\dots i_m}&=0, \quad \text{ for all } p=\{1, \dots, N \}\\
	\xi_{0}^{i_m} f_{i_1\dots i_m}&=0, \quad  \text{ for  } 1\leq  i_1\leq\dots \leq i_{m-1}
	\leq n.	  \end{align}
Using conditions on $\theta_p$ and $\xi_0$, it is proved in \cite{Krishnan2016} that above system of equations will imply $f=0$.
\end{proof}
\noindent For the more general case, when $\delta f$ is microlocally analytic at $(x_{0},\xi_{0})$, we use the same arguments as above, except that we multiply \eqref{eq:11} by an appropriate cut-off near $(x_{0},x_{0},\xi_{0})$ and use integration by parts as explained in \cite[Section 4]{KS} to conclude the following proposition:
\begin{prop}\label{Prop:Venky's Analogue}
	Let $\widetilde{\O}$, $f$ and  $\gamma_{0}$ be as in the statement of Proposition \ref*{S-U analogue}. If $(x_{0},\xi_{0})\notin WF_{A}(\delta f)$ (where $\xi_{0}$ is normal to the geodesic $\gamma_{0}$ at $x_0$), and  $I^{0}f(\gamma)=0$ for all $\gamma$ in a nbd. of $\gamma_{0}$, then $(x_{0},\xi_{0}) \notin WF_A(f)$.
\end{prop}

\noindent The rest of the argument from \cite{KS} applies as it is and thereby we prove Theorem \ref{Th:Extension theorem}. We will briefly outline the ideas here for the sake of completeness:
We will first need to show that the following analogue of \cite[Theorem 2.2(a)]{KS} holds for the case of symmetric $m$ tensor fields as well:
\begin{theorem}\label{KS_Th1}
Let $f$ be as above. Then $I^{0}f(\gamma) = 0$ for each geodesic $\gamma$ in $\mathcal{A}$, if and only if for each geodesic $\gamma_{0} \in \mathcal{A}$ there exists a a neighborhood $\mathcal{U}$ of $\gamma_0$ and a $(m-1)$-tensor field $v\in \mathcal{D}^\prime(\widetilde{\Omega_{\mathcal{U}}})$  such that $f = \D v$ in $\widetilde{\Omega_{\mathcal{U}}}$, and $v = 0$ outside $\Omega$.	
\end{theorem}
\noindent The ``if" part follows from the Fundamental Theorem of Calculus. To prove the ``only if" part of the theorem assume that $\gamma_{0}$ is a geodesic in the set $\mathcal{A}$, where $\mathcal{A}$ is defined in  Section \ref{Section: Preliminaries}. This means that it can be continuously deformed within the set to a point. Hence by extending all geodesics in $ \Omega$ to maximal geodesics in $\tilde {\Omega}$, we know that there must exist two continuous curves $a(t)\text{, }b(t)$, $t\in [0,1]$ such that $\gamma_{(a(0),b(0))}$ is tangent to $\partial \Omega$,   $\gamma_{(a(t),b(t))} \in \mathcal{A}$ and $\gamma_{(a(1),b(1))}$ is $\gamma_{0}$. Using \cite[Theorem A]{Morrey1957}, one can show that the Theorem \ref{KS_Th1} is at least true in a small neighborhood of $\partial \Omega$ i.e. in some neighborhood of the geodesics $\gamma_{(a(t),b(t))}$ for $0 \leq t \leq2 t_0$ for some $t_0<<1$. More precisely, 
\begin{Lemma}\cite [Lemma 5.1]{KS} \label{h in a boundary neighnorhood}
There exists a neighborhood $V$  of $\partial \Omega$ such that $\forall x \in V$, $dist(x,\partial \Omega)<\epsilon_0$ for some $\epsilon_{0} > 0$ and a unique $v_0$ such that $f=\D v_{0}$ in $V$, $v_0 = 0$ on $\partial \Omega$ and $v_0$ is analytic in $V$, up to the boundary $\partial \Omega$.
\end{Lemma}
\noindent Note that the above implies that in $V$, the tensor $h=f-\D v$ as constructed in Proposition \ref{prop:Construction of ODE} is zero. We will now construct a sequence of neighborhoods beginning with a neighborhood of $\gamma_{(a(0),b(0))}$ and up to a neighborhood of $\gamma_{(a(1),b(1))}$ for which the locally defined tensor field $h=f-\D v$ is zero. However to implement this program we will need the following theorem due to Sato-Kawai-Kashiwara, see e.g. \cite{Sato1973} or \cite{Zhou2000}:
\begin{Lemma}\cite[Lemma 3.1]{Zhou2000}\label{S-K-K}
Let $f \in \mathcal{D}^{\prime}(\Omega)$. Let $x_{0}\in \Omega$ and let $U$ be a neighborhood of $x_0$. Assume that $S$ is a $C^2$ submanifold of $\Omega$ and $x_0 \in supp(f) \cap S$. Furthermore, let $S$ divide $U$ into two open connected sets and assume that $f=0$ on one of these open sets. Let $\xi \in N_{x_0}^{*}(S)\setminus 0$, then $(x_0,\xi) \in WF_{A}(f)$.
\end{Lemma}

\noindent Consider the cone of all vectors in $T_{a(t)}\widetilde{\Omega}$ at an angle less than $\epsilon$ with $ \dot{\gamma}_{[a(t),b(t)]}$ for some small properly chosen $\epsilon$. The cone  $C_{\epsilon}(t)$ with its vertex at $a(t) \in \partial \widetilde{\Omega}$ is then the image of the above cone of vectors under the exponential map. We will choose $\epsilon > 0$ such that 
\begin{enumerate}
\item $C_{2\epsilon}(t) \subset \widetilde{\Omega}_{\mathcal{A}}$, $\forall t \in [0,1]$. 
\item $C_{\epsilon}(t) \subset \tilde{V}$ for $0 \leq t \leq t_0$ where $\tilde{V} := V \cup (\widetilde{\Omega}/\Omega)$.
\item No geodesic inside the cone $cl({C}_{2\epsilon}(t)) $, $t_{0} < t < 1$, with vertex at $a(t)$ is tangent to $\partial \Omega$. 

\end{enumerate}

\noindent For any $t$, let us construct a tensor field $h_t$ in $C_{\epsilon}(2t)$ just as in Proposition \ref{prop:Construction of ODE}. Recall that the support of  $h_t$ lies in  $\Omega$. 
Since $C_{\epsilon}(t) \subset \tilde{V}$ for $0 \leq t \leq t_0$ then by Lemma \ref{h in a boundary neighnorhood} we have $h_t = 0$ in $C_{\epsilon}(t) \subset \tilde{V}$. Hence the set $\{ t\in [0,1]: h_t = 0 \text{ in } C_{\epsilon}(t) \}$ is non empty.
Let $t^*=\text{sup} \{ t\in [0,1]: h_t = 0 \text{ in } C_{\epsilon}(t) \}$. We will show: $t^* = 1$. This will imply that there exists a neighborhood $\mathcal{U}$ of $\gamma_{0}$ and a $(m-1)$ tensor field $v \in \mathcal{D}^{\prime}(\widetilde{\Omega_{\mathcal{U}}})$ such that $h=f-\D v=0$ there.

\noindent  Assume $t^*<1$. Then $h_{t^* }= 0$ in $C_{\epsilon}(t^*)$ because $h_{t^*} = 0$ outside $\Omega$. Next we will show that $h_{t^*} = 0$ in $C_{2\epsilon}(t^*)$. This gives us a contradiction, because on increasing $t^*$ slightly to $t$, we can get $C_{\epsilon}(t) \cap \Omega \subset C_{2\epsilon} (t^*) \cap \Omega $ such that $h_t$ is zero in this $C_{\epsilon}(t)$. Here we would like to mention that as $h_t$ is got by solving an Initial Value Problem for a system of ODEs, hence they are locally unique. In particular, if $h_{t^*} = 0 \text{ in }C_{2\epsilon}(t^*)$ and $C_{\epsilon}(t) \cap \Omega \subset C_{2\epsilon} (t^*) \cap \Omega $, then $h_{t} = 0 \text{ in }C_{\epsilon}(t)$ which contradicts the choice for $t^{*}$. To fulfill our program, consider $h_{t^*}$ in $C_{2\epsilon} (t^*)$. As stated earlier, $h_{t^*}=0$ in  $C_\epsilon(t^*)$. Let $\epsilon<\epsilon_{0}\leq2 \epsilon$ be such that $C_{\epsilon_{0}}(t^*)$ is the first cone whose boundary intersects $supp(h_{t^*})$. If no such $\epsilon_0$ can be found then we are done. Let $q \in supp(h_{t^*})\cap \partial C_{\epsilon_0}(t)$. Clearly $q \notin \partial \widetilde{\Omega}$, because $h_{t^*} = 0$ outside $\Omega$. So $q$ is an interior point of $\widetilde{\Omega}$.
\noindent  In $\widetilde{\Omega}$, $(\delta f)_{i_1\dots i_{m-1}} = (\delta(f\chi))_{{i_1\dots i_{m-1}}}$ where $\chi$ is the characteristic function of $\Omega$. But one can first prove the theorem for $f$ such that $f = f^s$ in $\Omega$ and then make the argument for any general $f$. Working first with such tensor fields for which $f=f^s$, one knows that such a tensor field is analytic in $\partial \Omega$ up to $\partial \Omega$, see \cite[Section 5]{KS}.

\noindent Now, \begin{align*}
(\delta(f^s\chi))_{i_1\dots i_{m-1}} &= \left(\nabla_k(f^s_{i_1\dots i_{m-1}j} \chi)  \right)g^{jk} \\
&= \left( \chi \nabla_k f^s_{i_1\dots i_{m-1}j} \right)g^{jk} + f^s_{i_1\dots i_{m-1}j}g^{jk} \nabla_k \chi \\
&= f^s_{i_1\dots i_{m-1}j} \nabla^j\chi\\
&= -f^s_{i_1\dots i_{m-1}j} \nu^j\delta_{\partial \O}, \quad \text{ here }\partial_{\partial \O} \text{ represents dirac delta concentrated at }\partial \O.
\end{align*} 
This shows that the analytic wavefront set of $\delta f$ is in $N^*(\partial \Omega)$. Let $\tilde{\gamma}$ be the geodesic in $\O$ on the surface of $\partial C_{\epsilon_0}(t^*)$ that contains $q$. Because $N^*\tilde{\gamma}$ does not intersect $N^*\partial \Omega$, by Proposition \ref{Prop:Venky's Analogue} and by Lemma \ref{Lemma:Wavefront}, $h$ has no analytic singularities in $N^*\tilde{\gamma}$. Consider a small open set $W$ containing $q$ which is divided by the surface of $\partial C_{\epsilon_0}(t^*)$ into two open connected sets as in the statement of Lemma \ref{S-K-K} and $h_{t^*} = 0$ in one of these open sets. Since the co-normals to $C_{\epsilon_0} (t^*)$ at $q$ are not in $WF_{A}(h_{t^*})$, this implies $q \notin supp(h_{t^*})$ by the Sato-Kawai-Kashiwara theorem mentioned above. This shows that $h_{t^*}=0 \text{ in } C_{2\epsilon} (t^*)$ which in turn implies $t^* = 1$. This proves Lemma \ref{KS_Th1}.

\noindent Using the condition that any closed path with a base point on $\partial \Omega$ is homotopic to a point lying on $\partial \Omega$ and using the geometric arguments in Section 6 of \cite{KS} along with Lemma \ref{KS_Th1}, we conclude the proof of Theorem \ref{Th:Extension theorem}.

\noindent\textbf{Remark: } The symmetric $m-1$ tensor field $v$ also has components in $\mathcal{E}^\prime(\widetilde{\O})$ and is supported in $\Omega$ just like the $m$-tensor field $f$.

\subsection{Proof of Theorem \ref{Th:main Theorem}}
\begin{proof}[Proof of Theorem \ref{Th:main Theorem}]
	We will first prove the following lemma:
	\begin{Lemma}\label{Lemma:iteration}
		For any $ 1\leq k\leq m$, if $f = \D v$ with $v|_{\partial \O} =0$. Then   $I^{k}f =-kI^{k-1}v $.
	\end{Lemma}
	\begin{proof}
		Consider 
		\begin{align*}
			I^kf(\gamma) &= I^k(\D v)(\gamma)\\
			&= \int_0^{l(\gamma)}t^k(\D v)_{i_1\dots i_m}(\gamma(t))\dot{\gamma}^{i_1}(t)\dots \dot{\gamma}^{i_m}(t)dt\\
			&= \int_0^{l(\gamma)}t^{k} \frac{d}{dt}\{v_{i_1\dots i_{m-1}}(\gamma(t))\dot{\gamma}^{i_1}(t)\dots \dot{\gamma}^{i_{m-1}}(t)\}dt\\
			&=\{t^k v_{i_1\dots i_{m-1}}(\gamma(t))\dot{\gamma}^{i_1}(t)\dots \dot{\gamma}^{i_{m-1}}(t)\}_0^{l(\gamma)} \\
			&\quad \quad - k\int_0^{l(\gamma)} t^{k-1}v_{i_1\dots i_{m-1}}(\gamma(t))\dot{\gamma}^{i_1}(t)\dots \dot{\gamma}^{i_{m-1}}(t)dt\\
			&= -kI^{k-1}v(\gamma),
		\end{align*}
		where first term in the second last equality is $0$ because of our assumption $v|_{\partial \O}=0$. Thus, we have our lemma. 
	\end{proof}
	\noindent Let us come back to the proof of Theorem \ref{Th:main Theorem}. As we know from Theorem \ref{Th:Extension theorem} that if $I^0f(\gamma)=If(\gamma) = 0$ for each geodesic $\gamma$ not intersecting $K$ then there exist $(m-1)$-tensor field $v_1$ which is $0$ on the boundary $\partial \O$ such that $f =\D v_1$ on $\O\setminus K$. And from the Lemma \ref{Lemma:iteration}, we know 
	$$I^1f(\gamma) = I^1(\D v_1)(\gamma) = -I^0v_1(\gamma). $$
	\par \noindent Again using Theorem \ref{Th:Extension theorem} we conclude that there exist $(m-2)-$tensor field $v_2$ such that $v_1 = \D v_2$ and $v_2|_{\partial \O}=0$.  Using Theorem \ref{Th:Extension theorem} along with Lemma \ref{Lemma:iteration} $(m-2)$ more times, we have
	\begin{align*}
		I^mf(\gamma) = m!(-1)^m I^0v_m(\gamma) = 0
	\end{align*}
	where $v_m$ is $0$-tensor i.e. a function. Now using \cite[Theorem 1]{K1}, we can conclude $v_m =0$ on $\O\setminus K$.
	And since $f = \D^mv_m$ on  an open connected set $\O\setminus K$ therefore $f$ is also $0$ on $\O\setminus K$. 
\end{proof}
 \section{Proof of Theorem \ref{Th:Injectivity}}
 To prove Theorem \ref{Th:Injectivity}, we will need the $s$-injectivity of the ray transform for symmetric $m$-tensor fields. The proof of $s$-injectivity for symmetric $2$-tensor fields is given in \cite{SU2}. The same proof will also work for a symmetric tensor field of any order. For details, we will refer the reader to  \cite[Sections 2,3,4]{SU2}. Hence we have,
 \begin{theorem}\cite[Theorem 1.4]{SU2}
 	Let $(\O ,g)$ be a compact simple real analytic manifold with smooth boundary and $f$ be a symmetric $m$-tensor field with components in $L^2(\O)$. If $I^{0}f(\gamma)=0$ for all $\gamma$ which are geodesics in $\O$, then $f^{s}=0$ in $\O$.
 \end{theorem}

 \begin{theorem}\label{Decomposition Result}
  Let $\O$ be a compact simple Riemannian manifold with boundary. Let $m\geq0$ and $p\geq m$ be integers. Then for any $f \in L^2(S^m(\O))$, there exist uniquely determined $v_0, \dots , v_m$ with $v_i \in H^i(S^{m-i}\O)$ for $i = 0,1, \dots, m$ such that
 	\begin{align*}
 	f = \sum_{i=0}^m \D^iv_i, \quad \mbox{ with }  v_{i} \text{ solenoidal for } 0\leq i \leq  m-1 \\ \text{ and for each }  0\leq i \leq m-1, \quad 
 	\sum_{j=0}^i \D ^jv_{m-i+j} =0 \text{ on } \partial \Omega. \end{align*}
 \end{theorem}
 \begin{proof}
	This follows from a repeated application of \cite[Theorem 3.3.2]{Sharafutdinov_Book}.
\end{proof}

 \begin{proof}[Proof of Theorem \ref{Th:Injectivity}]
 	We have from Theorem \ref{Decomposition Result} that 
 	\begin{align}\label{Boundary values}
 	& \notag  f = \sum_{i=0}^m \D^iv_i, \quad \mbox{ with }  v_{i} \text{ solenoidal for } 0\leq i \leq  m-1 \\ 
 	& \text{ and for each }  0\leq i \leq m-1, \quad 
 	\sum_{j=0}^i \D ^jv_{m-i+j} =0 \text{ on } \partial \Omega. 
 	\end{align}
 	Using $s$-injectivity of $I$, we know that $v_{0} =0$, since it is solenoidal. Now consider
 	\begin{align*}
 0=	I^1f(\gamma) &= I^1\left(\sum_{i=0}^m \D^iv_i\right)(\gamma)\\
 	&= I^1\left(\D\left(\sum_{i=1}^m \D^{i-1}v_i\right)\right)(\gamma), \quad \quad \text{ since }v_0 =0\\
 	&=-I^0\left(\sum_{i=1}^m \D^{i-1}v_i\right)(\gamma) \quad \quad \text{(using Lemma \ref{Lemma:iteration})}
 	\end{align*}
 	From this, we can conclude $v_1$ is also $0$ because it is solenoidal part of tensor field $\sum_{i=1}^m \D^{i-1}v_i$. 
 	\par Now suppose that $v_{1},\cdots,v_{k}$ can be shown to be equal to $0$ from the knowledge of $I^{1}f,\cdots,I^{k}f$. Then
 	\begin{align*}
 0=	I^{k+1}\lb f-\sum\limits_{0}^{k} \D^{i} v_{i}\rb &=I^{k+1}\lb \sum\limits_{i=k+1}^{m} \D^{i} v_{i}\rb\\
 	\Rightarrow \hspace{3.2cm} I^{k+1}\lb \sum\limits_{i=k+1}^{m} \D^{i} v_{i}\rb&=  0\\
 		\Rightarrow (-1)^{k+1}(k+1)! I^0\left(\sum_{i=k+1}^m \D^{i-k-1}v_i\right)&=  0, \quad \text{(using Lemma \ref{Lemma:iteration},}\ \  (k+1)\ \  \text{times}).
 	\end{align*}
  Therefore $v_{k+1}=0$  because it is the  solenoidal part of the  tensor field $\left(\sum_{i=k+1}^m \D^{i-k-1}v_i\right)$. By induction, the proof is now complete. 
 \end{proof}

%
%

\section{Appendix}

\begin{proof}[Proof of Lemma \ref{Lemma:expansion of dv}]
First, let us recall for a $(m-1)$-tensor field $v$,
\begin{align*}
(\D v)_{i_1\dots i_m} = \sigma(i_1,\dots,i_m)\left(\frac{\partial v_{i_1\dots i_{m-1}}}{\partial x^{i_m}}-\sum_{l=1}^{m-1}\Gamma_{i_m i_l}^pv_{i_1,\dots i_{l-1} p i_{l+1}\dots i_{m-1}}\right).
\end{align*}
We will prove this result for $k=0,1,2$ and then for general $k\leq m$. 
\begin{align*}
(\D v)_{n\dots n } &=\frac{\partial v_{n\dots n}}{\partial x^{n}},\quad  \mbox{for } k=0\\
(\D v)_{n\dots n i} &= \frac{m-1}{m}\frac{\partial v_{n\dots n i }}{\partial x^n}-\frac{2(m-1)}{m}\Gamma_{in}^pv_{n\dots n p}+\frac{1}{m}\frac{\partial v_{n\dots n }}{\partial x^i},\quad  \mbox{for } k=1
\end{align*}
And for $k=2$, we have
\begin{align*}
(\D v)_{n \dots n ij}&= \sigma(n,\dots,n,i,j)\left(\frac{\partial v_{n\dots ni}}{\partial x^j}-\Gamma_{ij}^pv_{n\dots np} - (m-2)\Gamma_{nj}^p v_{n\dots n ip}\right)\\
&= \frac{\sigma(n,\dots, n , i)}{m}\left\{(m-1) \frac{\partial v_{ n\dots n i j}}{\partial x^n}+\frac{\partial v_{n\dots n i}}{\partial x^j}-2\Gamma_{ij}^pv_{n\dots np}-(m-2)\Gamma_{in}^p v_{n\dots n pj} \right. \\
& \ \ \ \left.-2(m-2)\Gamma_{nj}^p v_{n\dots n ip}- (m-2)^2\Gamma_{in}^p v_{n\dots n pj}   \right\}\\
&= \frac{\sigma(n,\dots, n , i)}{m}\left\{(m-1) \frac{\partial v_{ n\dots n i j}}{\partial x^n}+\frac{\partial v_{n\dots n i}}{\partial x^j}-2\Gamma_{ij}^pv_{n\dots np}-2(m-2)\Gamma_{nj}^p v_{n\dots n ip}\right. \\
& \ \ \ \left.- (m-1)(m-2)\Gamma_{in}^p v_{n\dots n pj}   \right\}\\
&= \frac{1}{m(m-1)}\left\{(m-1)\left( (m-2)\frac{\partial v_{ n\dots n i j}}{\partial x^n}+\frac{\partial v_{n\dots n j}}{\partial x^i}\right)+(m-1)\frac{\partial v_{n\dots n i}}{\partial x^j}\right.\\
&\quad  - 2 (\Gamma_{ij}^pv_{n\dots np}+(m-2)\Gamma_{nj}^pv_{n\dots npi})- 2(m-2)\left(\Gamma_{ij}^p v_{n\dots n p}+(m-2)\Gamma_{nj}^p v_{n\dots n pi}\right)\\
&\quad -  \left.2(m-1)(m-2)\Gamma_{in}^p v_{n\dots n pj}\right\}\\
&= \frac{1}{m(m-1)}\left\{(m-1)(m-2)\frac{\partial v_{ n\dots n i j}}{\partial x^n}+(m-1)\frac{\partial v_{n\dots n j}}{\partial x^i}+(m-1)\frac{\partial v_{n\dots n i}}{\partial x^j}\right.\\
&\quad  - \left.2 (m-1)\Gamma_{ij}^pv_{n\dots np}- 2(m-1)(m-2)\Gamma_{nj}^pv_{n\dots npi}  -2(m-1)(m-2)\Gamma_{in}^p v_{n\dots n pj}\right\}\\ 
&= \frac{1}{m}\left\{(m-2)\frac{\partial v_{ n\dots n i j}}{\partial x^n}+\frac{\partial v_{n\dots n j}}{\partial x^i}+\frac{\partial v_{n\dots n i}}{\partial x^j} - 2(m-2)\Gamma_{nj}^pv_{n\dots npi}  -2(m-2)\Gamma_{in}^p v_{n\dots n pj}\right.\\
& \quad \left.-2 \Gamma_{ij}^pv_{n\dots np}\right\}\\ 
&= \frac{m-2}{m}
\left\{ \frac{\partial v_{ n\dots n i j}}{\partial x^n}- 2\Gamma_{nj}^pv_{n\dots npi}  -2\Gamma_{in}^p v_{n\dots n pj}\right\}+ \frac{1}{m}\left\{ \frac{\partial v_{n\dots n j}}{\partial x^i}+\frac{\partial v_{n\dots n i}}{\partial x^j} -2 \Gamma_{ij}^pv_{n\dots np}\right\}.
\end{align*}
From above, we see that the result is true for $k=0, 1$ and $2$. Now, we are going to prove that the result is also true for $k\leq m$. Consider 
\begin{align*}
(\D v)_{n\dots n i_k\dots i_1} &= \sigma(n, \dots n ,i_k ,\dots, i_1)\left( \frac{\partial v_{n\dots n i_k\dots i_2}}{\partial x^{i_{1}}}- \sum_{l=2}^{k} \Gamma_{i_li_{1}}^p v_{n\dots n i_k\dots\hat{i_l} \dots i_2p}- (m-k)\Gamma_{n i_{1}}^pv_{n\dots n i_k\dots i_2p} \right)\\
&=  J+J^1_{k}+ (m-k)J^2_{k}.
\end{align*} 
where
\begin{align*}
J &=\sigma(n, \dots n ,i_k ,\dots, i_1)\left( \frac{\partial v_{n\dots n i_k\dots i_2}}{\partial x^{i_{1}}}\right),\\
J^1_{k} &= \sigma(n, \dots n ,i_k ,\dots, i_1)\left(  \sum_{l=2}^{k} \Gamma_{i_li_{1}}^p v_{n\dots n i_k\dots\hat{i_l} \dots i_2p}\right),\\
\mbox{ and }\quad  J^2_{k} &= \sigma(n, \dots n ,i_k ,\dots, i_1)\left(\Gamma_{n i_{1}}^pv_{n\dots n i_k\dots i_2p}\right).
\end{align*}
\begin{align*}
J &= \sigma(n, \dots n ,i_k ,\dots, i_1)\left(\frac{\partial v_{n\dots n i_k\dots i_2}}{\partial x^{i_1}}\right)\\
& = \frac{\sigma(n, \dots n ,i_k ,\dots, i_2)}{m}\left(\frac{\partial v_{n\dots n i_k\dots i_2}}{\partial x^{i_1}} + (m-1)\frac{\partial v_{n\dots n i_k\dots i_1}}{\partial x^{n}}\right)\\
&= \frac{1}{m}\frac{\partial v_{n\dots n i_k\dots i_2}}{\partial x^{i_1}} +  \frac{\sigma(n, \dots n ,i_k ,\dots, i_{3})}{m}\left(\frac{\partial v_{n\dots n i_k\dots i_{3}i_{1}}}{\partial x^{i_{2}}}+(m-2)\frac{\partial v_{n\dots n i_k\dots i_{1}}}{\partial x^{n}}\right)\\
&= \frac{1}{m}\sum_{l=1}^{k}\frac{\partial v_{n\dots n i_k\dots i_{l-1} i_{l+1}\dots  i_{1}}}{\partial x^{i_l}}+ \frac{m-k}{m}\frac{\partial v_{n\dots n i_k\dots i_{1}}}{\partial x^{n}},\quad \text{ repeating similar arguments }.
\end{align*}
		\begin{align*}
	J^2_{k} &= \sigma(n, \dots n ,i_k ,\dots, i_1)\left(\Gamma_{n i_{1}}^pv_{n\dots n i_k\dots i_2p}\right)\\
	&= \frac{\sigma(n, \dots n ,i_k ,\dots, i_2)}{m}\left(2\Gamma_{n i_{1}}^pv_{n\dots n i_k\dots i_2p}+ (m-2) \Gamma_{n i_{2}}^pv_{n\dots n i_k\dots i_3 i_1p}\right)\\
	&= \frac{2\sigma(n, \dots n ,i_k ,\dots, i_3)}{m(m-1)}\left(\Gamma_{i_2 i_{1}}^pv_{n\dots n i_k\dots i_3p}+(m-2)\Gamma_{ni_{1}}^pv_{n\dots n i_k\dots i_2p}\right)\\
	&\quad + \frac{(m-2)\sigma(n, \dots n ,i_k ,\dots, i_3)}{m(m-1)}\left(2 \Gamma_{n i_{2}}^pv_{n\dots n i_k\dots i_3 i_1p} + (m-3)\Gamma_{n i_{3}}^pv_{n\dots n i_k\dots i_4i_2 i_1p}\right)\\
	&= \frac{2}{m(m-1)}\Gamma_{i_2 i_{1}}^pv_{n\dots n i_k\dots i_3p}+\frac{2(m-2)\sigma(n, \dots n ,i_k ,\dots, i_3)}{m(m-1)}\sum_{q=1}^2\Gamma_{n i_{q}}^pv_{n\dots n i_k\dots i_3 \hat{i_q}i_1p}\\
	&\quad + \frac{(m-3)(m-2)\sigma(n, \dots n ,i_k ,\dots, i_3)}{m(m-1)}\Gamma_{n i_{3}}^pv_{n\dots n i_k\dots i_2 i_1p}\\
	&= \frac{2}{m(m-1)}\sum_{q,r=1,q\neq r}^3\Gamma_{i_q i_{r}}^pv_{n\dots n i_k\dots i_4\hat{i_q}\hat{i_r}i_1p}+\frac{2(m-3)\sigma(n, \dots n ,i_k ,\dots, i_4)}{m(m-1)}\sum_{q=1}^3\Gamma_{n i_{q}}^pv_{n\dots n i_k\dots i_4 \hat{i_q} i_1p}\\
	&\quad + \frac{(m-4)(m-3)\sigma(n, \dots n ,i_k ,\dots, i_4)}{m(m-1)}\Gamma_{n i_{4}}^pv_{n\dots n i_k\dots i_5i_3 i_2 i_1p},\\
	& \mbox{ repeating similar calculation  $(k-3)$  times, we get}\\
	&= \frac{2(k-1)}{m(m-1)}\sum_{q,r =1, q\neq r}^k\Gamma_{i_qi_{r}}^p v_{n\dots n i_k\dots\hat{i_q}\dots \hat{i_r}\dots i_1p}+\frac{2(m-k)}{m(m-1)}\sum_{q=1}^k\Gamma_{ni_q}^p v_{n\dots n i_k\dots\hat{i_q} \dots i_1p}. 
	\end{align*}
	\begin{align*}
	J^1_{k} &= \sigma(n, \dots n ,i_k ,\dots, i_1)\left(  \sum_{l=2}^{k} \Gamma_{i_li_{1}}^p v_{n\dots n i_k\dots\hat{i_l} \dots i_2p}\right)\\
	&= \frac{\sigma(n, \dots n ,i_k ,\dots, i_2)}{m}\left(2\sum_{l=2}^{k} \Gamma_{i_li_{1}}^p v_{n\dots n i_k\dots\hat{i_l} \dots i_2p}+(m-2)\sum_{l=3}^{k} \Gamma_{i_li_2}^p v_{n\dots n i_k\dots\hat{i_l} \dots i_3i_1p}\right.\\
	&\quad +\left.(m-2) \Gamma_{i_3i_{2}}^p v_{n\dots n i_k\dots i_4i_1p}\right)\\
	&= \frac{\sigma(n, \dots n ,i_k ,\dots, i_3)}{m(m-1)}\left\{2(k-1) \Gamma_{i_2i_{1}}^p v_{n\dots n i_k\dots i_3p}+(m-2)\left(2\sum_{l=3}^{k} \Gamma_{i_li_{1}}^p v_{n\dots n i_k\dots\hat{i_l} \dots i_2p}\right.\right.\\
	&\quad  +2\Gamma_{i_3i_{1}}^p v_{n\dots n i_k \dots i_4 i_2p}+  2\sum_{l=3}^{k} \Gamma_{i_li_2}^p v_{n\dots n i_k\dots\hat{i_l} \dots i_3i_1p}+(m-3)\sum_{l=4}^{k} \Gamma_{i_li_3}^p v_{n\dots n i_k\dots\hat{i_l} \dots i_4i_2 i_1p}\\
	&\quad \left.+ (m-3) \Gamma_{i_3i_4}^p v_{n\dots n i_k\dots i_5i_2i_1p}+2 \Gamma_{i_3i_{2}}^p v_{n\dots n i_k\dots i_4i_1p}+(m-3) \Gamma_{i_3i_{4}}^p v_{n\dots n i_k \dots i_5i_2i_1p}\Big{)}\right\}\\
	&= \frac{(m-2)\sigma(n, \dots n ,i_k ,\dots, i_3)}{m(m-1)}\left\{2\sum_{q=1}^2\left(\sum_{l=3}^{k}\Gamma_{i_li_q}^p v_{n\dots n i_k\dots\hat{i_l} \dots \hat{i_q}i_1p} +\Gamma_{i_3i_{q}}^p v_{n\dots n i_k \dots i_4\hat{i_q}i_1p} \right)\right.  \\
	&\quad  + \left.(m-3)\left(\sum_{l=4}^{k} \Gamma_{i_li_3}^p v_{n\dots n i_k\dots\hat{i_l} \dots i_4i_2i_1p}+2\Gamma_{i_3i_4}^p v_{n\dots n i_k \dots i_5i_2i_1p}\right)\right\}+\frac{2(k-1)}{m(m-1)} \Gamma_{i_2i_{1}}^p v_{n\dots n i_k\dots i_3p}\\
	&= \frac{(m-3)\sigma(n, \dots n ,i_k ,\dots, i_4)}{m(m-1)}\left\{2\sum_{q=1}^3\left(\sum_{l=4}^{k}\Gamma_{i_li_q}^p v_{n\dots n i_k\dots\hat{i_l} \dots \hat{i_q}i_1p}\right) +4\sum_{q=1}^3\left(\Gamma_{i_4i_{q}}^p v_{n\dots n i_k \dots \hat{i_q}i_1p} \right)\right.  \\
	&\quad  + \left.(m-4)\left(\sum_{l=5}^{k} \Gamma_{i_li_4}^p v_{n\dots n i_k\dots\hat{i_l} \dots i_1p}+3\Gamma_{i_5i_4}^p v_{n\dots n i_k \dots i_1p}\right)\right\}+\frac{2(k-1)}{m(m-1)}\sum_{q,r =1, q\neq r}^3\Gamma_{i_qi_{r}}^p v_{n\dots n i_k\dots\hat{i_q}\hat{i_r} i_1p}\\
	&\mbox{Repeating this expansion for $(k-2)$ times more to get}\\
	J^1_{k}  &= \frac{(m-k+1)\sigma(n, \dots, n ,i_{k})}{m(m-1)}\left\{2\sum_{q=1}^{k-1}\Gamma_{i_ki_q}^p v_{n\dots n i_{k-1}\dots\hat{i_l} \dots i_1p} +2(k-2)\sum_{q=1}^{k-1}\Gamma_{i_ki_q}^p v_{n\dots n i_{k-1}\dots\hat{i_l} \dots i_1p}\right.  \\
	&\quad  + (m-k)(k-1)\Gamma_{ni_k}^p v_{n\dots n i_{k-1} \dots i_1p}\Big{\}}+\frac{2(k-1)}{m(m-1)}\sum_{q,r =1, q\neq r}^{k-1}\Gamma_{i_qi_{r}}^p v_{n\dots n i_k\dots\hat{i_q}\hat{i_r} i_1p}\\
	&= \frac{(m-k+1)\sigma(n, \dots, n ,i_{k})}{m(m-1)}\left\{2(k-1)\sum_{q=1}^{k-1}\Gamma_{i_ki_q}^p v_{n\dots n i_{k-1}\dots\hat{i_q} \dots i_1p}\right.  \\
	&\quad  + (m-k)(k-1)\Gamma_{ni_k}^p v_{n\dots n i_{k-1} \dots i_1p}\Big{\}}+\frac{2(k-1)}{m(m-1)}\sum_{q,r =1, q\neq r}^{k-1}\Gamma_{i_qi_{r}}^p v_{n\dots n i_k\dots\hat{i_q}\hat{i_r} i_1p}\\
	&= \frac{2(k-1)}{m(m-1)}\sum_{q,r =1, q\neq r}^k\Gamma_{i_qi_{r}}^p v_{n\dots n i_k\dots\hat{i_q}\dots \hat{i_r}\dots i_1p}+\frac{2(k-1)(m-k)}{m(m-1)}\sum_{q=1}^k\Gamma_{ni_q}^p v_{n\dots n i_k\dots\hat{i_q} \dots i_1p}
	\end{align*}
	After putting the values of $ J, J^1_k$ and $J_k^2$ in $\D v$, we get
	\begin{align*}
	(\D v)_{n\dots n i_k\dots i_1} &=\frac{(m-k)}{m}\frac{\partial v_{n\dots ni_k\dots i_1}}{\partial x^n} -\frac{2(m-k)}{m} \sum_{l=1}^k\Gamma_{i_l n}^p v_{n\dots n i_k\dots \hat{i_l}\dots i_1 p}\\ 
	&\quad +\frac{1}{m}\sum_{l=1}^k \frac{\partial v_{n\dots ni_k\dots \hat{i_l}\dots  i_1}}{\partial x^{i_l}}-\frac{2}{m} \sum_{l,q=1, l \neq q}^k\Gamma_{i_l i_q}^p v_{n\dots n i_k\dots \hat{i_l}\dots\hat{i_q}\dots i_1 p}.
	\end{align*}
\end{proof}

\begin{proof}[Proof of estimate \eqref{eq:8}]
Let $^{t}L= \frac{\Phi_{\xi}.\partial_{\xi}}{i\lambda|\Phi_\xi|^{2}}$. Then as already noted
\begin{equation*}
^{t}L^{N}(e^{i\lambda \Phi(x,\xi)})=e^{i\lambda \Phi(x,\xi)}. 
\end{equation*} 
Consider, \begin{align*}
&\left|	\iint_{|x-y|>\delta/C_{0}} (^{t}L^{N}(e^{i\lambda \Phi(y,x,\xi,\eta)}))\tilde{\tilde{a_{N}}}(x,\xi)f_{i_{1}....i_{m}}(z){\tilde{b}}^{i_{1}}(x,\xi)...{\tilde{b}}^{i_{m}}(x,\xi)dxd\xi \right|\\& \quad \quad \leq \left|	\iint_{|x-y|>\delta/C_{0}} e^{i\lambda \Phi(y,x,\xi,\eta)}L^{N}(\tilde{\tilde{a_{N}}}(x,\xi)f_{i_{1}....i_{m}}(z){\tilde{b}}^{i_{1}}(x,\xi)...{\tilde{b}}^{i_{m}}(x,\xi))dxd\xi \right|\\ 
&\quad \quad \quad  + N\int_{|x-y|>\delta/C_{0}}e^{-\lambda \delta^{2}/2}\left|f_{i_{1}...i_{m}}(x)B^{i_{1}...i_{m}}(x,\xi_{bdry})\right|dx.
\end{align*}
Using the fact that, $f$ is compactly supported and using \eqref{eq:estimate on cutoff}, we get \eqref{eq:8}.
\end{proof}

\begin{proof}[Proof of the estimate \eqref{eq:10}]
Consider 
$$ \left|	\iint_{|x-y|<\delta/C_{0}} \left(e^{i\lambda \Phi(y,x,\xi,\eta)}\right)\tilde{\tilde{a_{N}}}(x,\xi)f_{i_{1}\dots
	i_{m}}(z){\tilde{b}}^{i_{1}}(x,\xi)\cdots{\tilde{b}}^{i_{m}}(x,\xi)dxd\xi \right|.$$
Rewrite the above as : 
$$ \left|	\int_{|x-y|<\delta/C_{0}}(e^{i\lambda \Phi(y,x,\xi_{c},\eta)})(e^{-i\lambda \Phi(y,x,\xi_{c},\eta)})\int_{|\xi-\eta|<\delta/C_{0}} (e^{i\lambda \Phi(y,x,\xi,\eta)})\tilde{\tilde{a_{N}}}(x,\xi)f_{i_{1}\dots i_{m}}(z){\tilde{b}}^{i_{1}}(x,\xi)\cdots{\tilde{b}}^{i_{m}}(x,\xi)dxd\xi \right|$$
Using (2.10) of \cite{Sjostrand}  to the above, we get
\begin{align*}
&\left|	\int_{|x-y|<\delta/C_{0}} (e^{i\lambda \Phi(y,x,\xi_{c},\eta)})f_{i_{1}\dots i_{m}}(x)\sum_{0\leq k \leq \lambda /C}C_{n}\frac{1}{k!}\lambda^{-n/2-k}\frac{(\bigtriangleup^{k})}{2}\big(\tilde{\tilde{a_{N}}}(x,\xi_{c}){\tilde{b}}^{i_{1}}(x,\xi_{c})\cdots{\tilde{b}}^{i_{m}}(x,\xi_{c})\big)\right.\\
&\quad \quad \quad +R(x,y,\eta,\lambda)dx \big{|}
\end{align*}
\begin{Lemma}
	$$ \sum_{0\leq k \leq \lambda /C}C_{n}\frac{1}{k!}\lambda^{-n/2-k}\frac{(\bigtriangleup^{k})}{2}\big(\tilde{\tilde{a_{N}}}(x,\xi_{c}){\tilde{b}}^{i_{1}}(x,\xi_{c})\cdots{\tilde{b}}^{i_{m}}(x,\xi_{c})\big)$$ is a formal analytic symbol.
\end{Lemma}
\begin{proof}
	Let, $$a_{k} = \frac{1}{k!}\frac{(\bigtriangleup^{k})}{2}\big(\tilde{\tilde{a_{N}}}(x,\xi_{c}){\tilde{b}}^{i_{1}}(x,\xi_{c})\cdots{\tilde{b}}^{i_{m}}(x,\xi_{c})\big)$$
	Then from Cauchy integral formula  \cite [Section 2.4]{Sjostrand},\\
	\begin{align*}
	|a_{k}|&\leq C_{n}(k+1)^{n/2}(k-1)! 2^{k}\sup_{B(\xi_c)}\big(\tilde{\tilde{a_{N}}}(x,\xi_{c}){\tilde{b}}^{i_{1}}(x,\xi_{c})...{\tilde{b}}^{i_{m}}(x,\xi_{c})\big)\\
	& \leq C1_{n}(k+1)^{n/2}(k-1)! 2^{k}\\
	& \leq  C2_{n}(k+1)^{n/2}e^{2-k}(k-1)^{k-1/2} 2^{k} \text{ (Using Stirling's approximation)}\\
	& \leq C2_{n}\bigg(\frac{2}{e}\bigg)^{k+1}(k+1)^{n/2+k}\\
	& \leq \tilde{C_{n}}^{k+n/2}(k+n/2)^{n/2+k}.
	\end{align*}
	Hence,  $$ \sum_{0\leq k \leq \lambda /C}C_{n}\frac{1}{k!}\lambda^{-n/2-k}\frac{(\bigtriangleup^{k})}{2}\big(\tilde{\tilde{a_{N}}}(x,\xi_{c}){\tilde{b}}^{i_{1}}(x,\xi_{c})...{\tilde{b}}^{i_{m}}(x,\xi_{c})\big)  =   \sum_{0\leq k \leq \lambda /C}\lambda^{-n/2-k}a_{k+n/2}$$ is a formal analytic symbol ${{B}}^{i_{1}...i_{m}}(x,y,\eta;\lambda)$ by \cite [Excercise 1.1]{Sjostrand}.
\end{proof}
Hence,
\begin{align*}
&\int_{|x-y|<\delta/C_{0}} (e^{i\lambda \Phi(y,x,\xi,\eta)})\tilde{\tilde{a_{N}}}(x,\xi)f_{i_{1}\dots i_{m}}(z){\tilde{b}}^{i_{1}}(x,\xi)\cdots {\tilde{b}}^{i_{m}}(x,\xi)dxd\xi\\ & \quad \quad= 	\int_{|x-y|<\delta/C_{0}} (e^{i\lambda \Phi(y,x,\xi_{c},\eta)})f_{i_{1}\dots i_{m}}(x){{B}}^{i_{1}\dots i_{m}}(x,y,\eta;\lambda)dx\\
&\quad \quad \quad +	\int_{|x-y|<\delta/C_{0}} (e^{i\lambda \Phi(y,x,\xi_{c},\eta)})f_{i_{1}\dots i_{m}}(x)R(x,y,\eta;\lambda)dxd\xi .
\end{align*}
But, $$ \left|\int_{|x-y|<\delta/C_{0}} (e^{i\lambda \Phi(y,x,\xi_{c},\eta)})f_{i_{1}\dots i_{m}}(x)R(x,y,\eta;\lambda)dx\right| =\mathcal{O}(e^{-\lambda/c}).$$ 
Since, 
$$|R(x,y,\eta;\lambda)|\leq \O/C e^{-\lambda/c}.$$(See 2.10, \cite{Sjostrand}). 
So, this along with \eqref{eq:7} and \eqref{eq:8}, gives us: \\
$$\left|\int_{|x-y|<\delta/C_{0}} (e^{i\lambda \Phi(y,x,\xi_{c},\eta)})f_{i_{1}\dots i_{m}}(x){{B}}^{i_{1}\dots i_{m}}(x,y,\eta;\lambda)dx\right| = \mathcal{O}(e^{-\lambda/c}).$$
\end{proof} 

\bibliographystyle{plain}
\bibliography{Refs}

\end{document}